\documentclass[11pt]{article}
\usepackage{amsfonts}
\usepackage{amsmath}

\newtheorem{theorem}{Theorem}

\newtheorem{corollary}[theorem]{Corollary}

\newtheorem{lemma}[theorem]{Lemma}

\newtheorem{remark}[theorem]{Remark}

\newenvironment{proof}{\noindent\\ \noindent\relax{\sc
     Proof}}{{\samepage\par\nopagebreak\hbox
     to\hsize{\hfill$\Box$}}}

\begin{document}

\title{Limit theorems for decomposable branching processes in a random
environment}
\author{Vladimir Vatutin\thanks{
Steklov Mathematical institute RAS, Gubkin str. 8, Moscow, 119991,
Russia;
e-mail: vatutin@mi.ras.ru},\, Quansheng Liu \thanks{%
Univ Bretagne - Sud, UMR 6205, LMBA, F-56000 Vannes, France;
e-mail: \uppercase{Q}uansheng.\uppercase{L}iu@univ-ubs.fr}}
\date{}
\maketitle

\begin{abstract}
We study the asymptotics of the survival probability for the
critical and decomposable branching processes in random
environment and prove Yaglom type limit theorems for these
processes. It is shown that such processes possess some properties
having no analogues for the decomposable branching processes in
constant environment
\end{abstract}



\section{Introduction}

The multitype branching processes in random environment we
consider here can be viewed as a discrete-time stochastic model
for the sizes of a geographically structured population occupying
islands labelled $0,1,...,N.$ One unit of time represents a
generation of particles (individuals). Particles located on island
$0$ give birth under influence of a randomly changing environment.
They may migrate to one of the islands $1,2,...,N$ immediately
after birth, with probabilities again depending upon the current
environmental state. Particles of island $i\in $ $\left\{
1,2,...,N-1\right\} $ either stay at the same island or migrate to
the islands $i+1,2,...,N$ and their reproduction laws are not
influenced by any changing environment. Finally, particles of
island $N$ do not migrate and evolve in a constant environment.

The goal of this paper is to investigate the asymptotic behavior
of the survival probability of the whole process and the
distribution of the number of particles in the population given
its survival or survival of particles of type 1.

Let $m_{ij}$ be the mean number of type $j$ particles produced by
a type $i$ particle at her death.

We formulate our main assumptions as

\textbf{Hypothesis }$A0:$

\begin{itemize}
\item particles of type $0$ form (on their own) a critical
branching process in a random environment;

\item particles of any type $i\in \left\{ 1,2,...,N\right\} $ form
(on their
own) a critical branching process in a constant environment, i.e., $m_{ii}=1$%
;

\item particles of any type $i$ are able to produce descendants of
all the next in order types (may be not as the direct descendants)
but not any
preceding ones. In particular, $m_{ij}=0$ for $1\leq j<i\leq N$ and $%
m_{i,i+1}>0$ for $i=1,...,N-1$.
\end{itemize}

Let $X_{n}$ be the number of particles of type $0$ and $\mathbf{Z}%
_{n}=\left( Z_{n1},...,Z_{nN}\right) $ be the vector of the
numbers of particles type $1,2,...,N$, respectively, present at
time $n$. Throughout of this paper considering the $( N+1)-$type
branching process it is
assumed (unless otherwise specified) that $X_{0}=1$ and $\mathbf{Z}%
_{0}=( 0...,0) =\mathbf{0}$.

We investigate asymptotics of the survival probability of this process as $%
n\rightarrow \infty $ and the distribution of the number of
particles in the process at moment $n$ given $Z_{n1}>0$ or
$\mathbf{Z}_{n}$ $\neq \mathbf{0}.$ Note that the asymptotic
behavior of the survival probability for the case $N=1$ has been
investigated in \cite{VDJS} under stronger assumptions than those
imposed in the present paper. The essential novelty of this paper
are Yaglom-type limit theorems for the population vector
$\mathbf{Z}_{n}$ (see Theorem \ref{T_main}\ below).

The structure of the remaining part of this paper is as follows. In Section %
\ref{Sfix1} we recall known facts for decomposable branching
processes in constant environments and show some preliminary
results. Section \ref{SecREnvironm} deals with the $(N+1)-$type
decomposable branching processes in random environment. Here we
study the asymptotic behavior of the survival probability and
prove a Yaglom-type conditional limit theorem for the number of
particles in the process given $Z_{n1}>0$. In Section
\ref{Secthree} we consider a 3$-$type decomposable branching
process in random environment and, proving a
Yaglom-type conditional limit theorem under the condition $Z_{n1}+Z_{n2}>0,$%
\textbf{\ }show the essential difference of such processes with
the decomposable processes evolving in constant environment.

\section{Multitype decomposable branching processes in a constant
environment \label{Sfix1}}

The aim of this section is to present a number of known results
about the decomposable branching processes we are interesting in
the case of a
constant environment and, therefore, we do not deal with particles of type $%
0 $.

If Hypothesis $A0$ is valid then the mean matrix of our process
has the form

\begin{equation}
\mathbf{M=}\left( m_{ij}\right) =\left(
\begin{array}{ccccc}
1 & m_{12} & ... & ... & m_{1N} \\
0 & 1 & m_{23} & ... & m_{2N} \\
0 & 0 & 1 & ... & ... \\
... & ... & ... & ... & ... \\
... & ... & ... & ... & m_{N-1,N} \\
0 & 0 & ... & 0 & 1%
\end{array}%
\right),   \label{matrix1}
\end{equation}%
where
\begin{equation}
m_{i,i+1}>0,\ i=1,2,...,N-1.  \label{Maseq}
\end{equation}%
Under conditions (\ref{matrix1}) and (\ref{Maseq}) one obtains a
complete ordering $1\longrightarrow 2\longrightarrow
...\longrightarrow N$ of types.

Observe that according to the classification given in \cite{FN2}
the process we consider is \textit{strongly critical.}

In the sequel we need some results from \cite{FN} and \cite{FN2}.
To this aim we introduce additional notation.

1) For any vector $\mathbf{s}=(s_{1},...,s_{p})$ (the dimension
will usually
be clear from the context), and integer valued vector $\mathbf{k}%
=(k_{1}.....k_{p})$ define%
\begin{equation*}
\mathbf{s}^{\mathbf{k}}=s_{1}^{k_{1}}...s_{p}^{k_{p}}.
\end{equation*}%
Further, let $\mathbf{1}=\left( 1,...,1\right) $ be a vector of
units and let $\mathbf{e}_{i}$ be a vector whose $i$-th component
is equal to one while the remaining are zeros.

2) The first and second moments of the components of the population vector $%
\mathbf{Z}_{n}=\left( Z_{n1},...,Z_{nN}\right) $ will be denoted as%
\begin{equation*}
m_{il}(n):=\mathbf{E}\left[
Z_{nl}|\mathbf{Z}_{0}=\mathbf{e}_{i}\right] ,\ m_{il}:=m_{il}(1),
\end{equation*}%
\begin{equation}
b_{ikl}(n):=\mathbf{E}\left[ Z_{nk}Z_{nl}-\delta _{kl}Z_{nl}|\mathbf{Z}_{0}=%
\mathbf{e}_{i}\right] ,\ b_{ikl}:=b_{ikl}(1).  \label{Defq}
\end{equation}

To go further we introduce probability generating functions

\begin{equation}
h^{(i,N)}(\mathbf{s}) :=\mathbb{E}\left[ \prod_{k=i}^N s_{k}^{\eta
_{ik}}\right], \quad 1\leq i\leq N, \label{DefNONimmigr}
\end{equation}%
where $\eta _{ij}$ represents the number of daughters of type $j$
of a mother of type $i\in\{1,2,...,N\}.$ Let
\begin{equation}
H_{n}^{(i,N)}(\mathbf{s}):=\mathbb{E}\left[ \prod_{k=i}^N
s_k^{Z_{nk}}|\mathbf{Z}_{0}=\mathbf{e}%
_{i}\right] ,\quad 1\leq i\leq N,
\end{equation}%
be the probability generating functions for the vector of the
number of particles at moment $n$ given the process is initiated
at time $0$ by a
singly particle of type $i\in \left\{ 1,2,...,N\right\} $ . Clearly, $%
H_{1}^{(i,N)}(\mathbf{s})=h^{(i,N)}(\mathbf{s})$. Denote
\begin{eqnarray*}
\mathbf{H}_{n}(\mathbf{s}) &:&=\left( H_{n}^{(1,N)}(\mathbf{s}%
),...,H_{n}^{(N,N)}(\mathbf{s})\right) , \\
\mathbf{Q}_{n}(\mathbf{s}) &:&=\left( Q_{n}^{(1,N)}(\mathbf{s}%
),...,Q_{n}^{(N,N)}(\mathbf{s})\right) =\left( 1-H_{n}^{(1,N)}(\mathbf{s}%
),...,1-H_{n}^{(N,N)}(\mathbf{s})\right) .
\end{eqnarray*}

As usually, for two sequences $a_n,b_n$ we write $a_n\sim b_n,
a_n=O(b_n),\,a_n=o(b_n)$ and $a_n\asymp b_n$ meaning that these
relationships are valid as $n\to\infty.$ In particular, $a_n\asymp
b_n$ if and only if
$$0<\liminf_{n\to\infty}a_n/b_n\leq\limsup_{n\to\infty}a_n/b_n<\infty.
$$

The following theorem is a simplified combination of the
respective results from \cite{FN} and \cite{FN2}:

\begin{theorem}
\label{T_Foster}Let $\left\{ \mathbf{Z}_{n},n=0,1,...\right\} $ be
a strongly critical multitype branching process satisfying
(\ref{matrix1}) and
(\ref{Maseq}). Then, as $n\rightarrow \infty $%
\begin{equation}
m_{il}(n)\sim c_{il}n^{l-i},\ i\leq l,  \label{MomentSingle3}
\end{equation}%
where $c_{il}$ are positive constants known explicitly (see
\cite{FN2}, Theorem 1);

2) if $b_{ikl}<\infty ,\ i,k,l=1,...,N$ then%
\begin{equation}
b_{ikl}(n)\sim c_{ikl}n^{k+l-2i+1},  \label{MomentSingle2}
\end{equation}%
where $c_{ikl}$ are constants known explicitly (see \cite{FN2},
Theorem 1) and\ \
\begin{equation}
Q_{n}^{(i,N)}(\mathbf{0})=1-H_{n}^{(i,N)}(\mathbf{0})=\mathbf{P}(\mathbf{Z}%
_{n}\neq \mathbf{0}|\mathbf{Z}_{0}=\mathbf{e}_{i})\sim
c_{i}n^{-2^{-(N-i)}},c_{i}>0.  \label{SurvivSingle}
\end{equation}
\end{theorem}

Let $H(s_{1},...,s_{p})=H(\mathbf{s})$ be a multivariate
probability generating function with
\begin{equation*}
m_{l}:=\frac{\partial H(\mathbf{s})}{\partial s_{l}}\left\vert _{\mathbf{s}=%
\mathbf{1}}\right. ,\quad b_{kl}:=\frac{\partial
^{2}H(\mathbf{s})}{\partial s_{k}\partial s_{l}}\left\vert
_{\mathbf{s}=\mathbf{1}}\right.<\infty .
\end{equation*}

\begin{lemma}
\label{L_approx}(see formula (1), page 189, in \cite{AN72}) For any $\mathbf{%
s}= (s_1, ..., s_p)\in[ 0,1] ^{p}$ we have%
\begin{equation*}
\sum_{l=1}^{p}m_{l}\left( 1-s_{l}\right) -\frac{1}{2}%
\sum_{k,l=1}^{p}b_{kl}\left( 1-s_{k}\right) \left( 1-s_{l}\right) \\
\leq 1-H(\mathbf{s})\leq \sum_{l=1}^{p}m_{l}\left( 1-s_{l}\right).
\end{equation*}
\end{lemma}

From now on we agree to denote by $C,C_{0},C_{1},...$ positive
constants which may be different in different formulas.

For $s= (s_1,...,s_N)$ put
\begin{equation}
M_{i}(n;\mathbf{s}):=\sum_{l=i}^{N}m_{il}(n)\left( 1-s_{l}\right),\, B_{i}(n;%
\mathbf{s}):=\frac{1}{2}\sum_{k,l=i}^{N}b_{ikl}(n)\left(
1-s_{k}\right) \left( 1-s_{l}\right) .  \label{DefAB}
\end{equation}

\begin{lemma}
\label{L_multit} Let the conditions of Theorem 1 be valid. Then for any
tuple $t_{1},...,t_{N}$ of positive numbers and
\begin{equation*}
1-s_{l}=n^{-t_{l}},\ l=1,2,...,N
\end{equation*}%
there exists  $C_{+}<\infty$ such that, for all $n=1,2,...$
\begin{equation*}
Q_{n}^{(i,N)}(\mathbf{s})\leq C_{+}\min \left\{
n^{-2^{-(N-i)}},n^{-\min_{i\leq l\leq N}\left( t_{l}-l+i\right) }\right\} .
\end{equation*}%
If, in addition,
\begin{equation}
\min_{i\leq l\leq N}\left( t_{l}-l+i\right)\geq 1  \label{RanT}
\end{equation}%
then there exists a positive constant $C_{-}$ such that, for all $n=1,2,...$
\begin{equation}
C_{-}n^{-\min_{i\leq l\leq N}\left( t_{l}-l+i\right) }\leq Q_{n}^{(i,N)}(%
\mathbf{s})\leq C_{+}n^{-\min_{i\leq l\leq N}\left( t_{l}-l+i\right) }.
\label{twoSide}
\end{equation}
\end{lemma}

\begin{proof}
Take $\varepsilon \in (0,1]$ and denote $\mathbf{s}(\varepsilon
)=\left( 1-\varepsilon n^{-t_{1}},...,1-\varepsilon n^{-t_{N}}\right) $. By
Lemma \ref{L_approx} and monotonicity of $Q_{n}^{(i,N)}(\mathbf{s}(\varepsilon ))$ in $\varepsilon,$ we have
\begin{equation}
M_{i}(n;\mathbf{s}(\varepsilon ))-B_{i}(n;\mathbf{s}(\varepsilon ))\leq
Q_{n}^{(i,N)}(\mathbf{s}(\varepsilon ))\leq Q_{n}^{(i,N)}(\mathbf{s})\leq
M_{i}(n;\mathbf{s}).  \label{Bou2}
\end{equation}%
In view of (\ref{MomentSingle3}) - (\ref{MomentSingle2}) there exist
positive constants $C_{j},j=1,2,3,4$ such that
\begin{eqnarray}
\varepsilon C_{1}n^{-\min_{i\leq l\leq N}\left( t_{l}-l+i\right) } &\leq
&\varepsilon C_{1}\sum_{l=i}^{N}\frac{n^{l-i}}{n^{t_{l}}}\leq M_{i}(n;%
\mathbf{s}(\varepsilon ))=\varepsilon \sum_{l=i}^{N}m_{il}(n)n^{-t_{l}}
\notag \\
&\leq &M_{i}(n;\mathbf{s})\leq C_{2}\sum_{l=i}^{N}\frac{n^{l-i}}{n^{t_{l}}}%
\leq C_{3}n^{-\min_{i\leq l\leq N}\left( t_{l}-l+i\right) }  \label{Bou3}
\end{eqnarray}%
and%
\begin{equation*}
0\leq B_{i}(n;\mathbf{s}(\varepsilon ))\leq \varepsilon
^{2}C_{4}\sum_{k,l=i}^{N}\frac{n^{k-i+1+l-i}}{n^{t_{k}}n^{t_{l}}}.
\end{equation*}
If now $\min_{i\leq k\leq N}\left( t_{k}-k+i-1\right) \geq 0,$ then for
a fixed $\varepsilon >0$%
\begin{equation}
0\leq B_{i}(n;\mathbf{s}(\varepsilon ))\leq \varepsilon
^{2}C_{4}\sum_{k,l=i}^{N}\frac{1}{n^{t_{l}-\left( l-i\right)
}n^{t_{k}-\left( k-i+1\right) }}\leq \varepsilon^2N^2
C_{4}n^{-\min_{i\leq l\leq N}\left( t_{l}-l+i\right) }.
\label{Bou4}
\end{equation}%
Take $0<\varepsilon<\min\{1,C_1/N^2C_4\}.$ Then the  estimates (\ref{Bou2})--(\ref{Bou4})   give (\ref{twoSide}) with $C_{-}= \varepsilon C_1 - \varepsilon^2 N^2C_4$ and $C_+ = C_3.$
\end{proof}

Write $\mathbf{0}^{(r)}=(0,0,...,0)$ and
$\mathbf{1}^{(r)}=(1,1,...,1)$ for the $r$-dimensional vectors all
whose components are zeros and ones, respectively; set $
\mathbf{s}_{r}=\left( s_{r},s_{r+1},...,s_{N}\right)$
and denote by $I\left\{ \mathcal{A}\right\} $  the indicator of the event $%
\mathcal{A}$.

The next lemma, in which we assume that
$\mathbf{Z}_{0}=\mathbf{e}_{1}$
gives an approximation for the function $Q_{n}^{(1,N)}(\mathbf{0}^{(r)},%
\mathbf{s}_{r+1})$.

\begin{lemma}
\label{L_singleterm}If $\min_{r+1\leq l\leq N}\left(
t_{l}-l+1\right)
>2^{-\left( r-1\right) }$ and
\begin{equation*}
1-s_{l}=n^{-t_{l}\text{ }},\,l=r+1,r+2,...,N,
\end{equation*}%
then, as $n\rightarrow \infty $%
\begin{equation*}
Q_{n}^{(1,N)}(\mathbf{0}^{(r)},\mathbf{s}_{r+1})\sim
\mathbf{P}\left( Z_{nr}>0\right) \sim c_{r}n^{-2^{-\left(
r-1\right) }}.
\end{equation*}
\end{lemma}

\begin{proof}
 In view of (\ref{SurvivSingle})  we have for
$\mathbf{s}_{r+1}\in [0,1]^{N-r}:$
\begin{eqnarray*}
\mathbf{P}\left( Z_{nr}>0\right) &\leq &\mathbf{P}\left( \cup
_{j=1}^{r}\left\{ Z_{nj}>0\right\} \right) =Q_{n}^{(1,N)}(\mathbf{0}^{(r)},%
\mathbf{1}^{(N-r)}) \\
&\leq
&Q_{n}^{(1,N)}(\mathbf{0}^{(r)},\mathbf{s}_{r+1})=\mathbf{E}\left[
1-s_{r+1}^{Z_{n,r+1}}...s_{N}^{Z_{nN}}I\left\{ \cap
_{j=1}^{r}\left\{
Z_{nj}=0\right\} \right\} \right] \\
&\leq &\mathbf{P}\left( \cup _{j=1}^{r}\left\{ Z_{nj}>0\right\} \right) +%
\mathbf{E}\left[ 1-s_{r+1}^{Z_{n,r+1}}...s_{N}^{Z_{nN}}\right] \\
&\leq &\sum {}_{j=1}^{r}\mathbf{P}\left( Z_{nj}>0\right)
+\mathbf{E}\left[
1-s_{r+1}^{Z_{n,r+1}}...s_{N}^{Z_{nN}}\right] \\
&=& \left( 1+o(1)\right)\mathbf{P}\left( Z_{nr}>0\right)
+Q_{n}^{(1,N)}\left( \mathbf{1}^{(r)},\mathbf{s}_{r+1}\right) .
\end{eqnarray*}%
Further, by the conditions of the lemma we deduce%
\begin{eqnarray*}
Q_{n}^{(1,N)}\left( \mathbf{1}^{(r)},\mathbf{s}_{r+1}\right) &\leq
&\sum_{l=r+1}^{N}m_{1l}(n)n^{-t_{l}} \\
&\leq &Cn^{-\min_{r+1\leq l\leq N}\left( t_{l}-l+1\right) }=o\Big(
n^{-2^{-\left( r-1\right) }}\Big).
\end{eqnarray*}%
 Hence the statement of the
lemma follows.
\end{proof}

\subsection{The case of two types}

Here we consider the situation of two types and investigate the
behavior of the function $1-H_{n}^{(1,2)}\left( s_{1},s_{2}\right)
$ as $n\rightarrow
\infty $ assuming that%
$
1-s_{i}=n^{-t_{i}},\quad i=1,2.
$

\begin{lemma}
\label{L_twodimen}If the conditions of Theorem 1 are valid for $N=2,$ then%
\begin{equation*}
1-H_{n}^{(1,2)}\left( s_{1},s_{2}\right) \asymp \left\{
\begin{array}{ccc}
n^{-1/2} & \text{if} & t_{1}\in (0,\infty),\,0<t_{2}\leq 1; \\
n^{-t_{2}/2} & \text{if} & t_{1}\in( 0,\infty),\,1<t_{2}<2;
\\
n^{-1} & \text{if} & 0<t_{1}< 1,\,t_{2}\geq 2; \\
n^{-1-\min \left( t_{1}-1,t_{2}-2\right) } & \text{if} & t_{1}\geq
1,\,t_{2}\geq 2.
\end{array}%
\right.
\end{equation*}
\end{lemma}

\begin{proof}
Observe that for any $0\leq s_{1}\leq s^{\prime}_{1}\leq 1$%
\begin{eqnarray}
H_{n}^{(1,2)}\left( s^{\prime}_{1},s_{2}\right)
-H_{n}^{(1,2)}\left( s_{1},s_{2}\right) &=&\mathbf{E}\left[ \left(
(s^{\prime}_{1})^{Z_{n1}}-s_{1}^{Z_{n1}}\right)
s_{2}^{Z_{n2}}\right]
\notag \\
&\leq &\mathbf{E}\left[ 1-s_{1}^{Z_{n1}}\right]
=1-H_{n}^{(1,1)}\left(
s_{1}\right)  \notag \\
&\leq &\mathbf{P}\left(
Z_{n1}>0|\mathbf{Z}_{0}=\mathbf{e}_{1}\right) \leq Cn^{-1}.
\label{Neglig}
\end{eqnarray}%
Let now $m=m(s_{2})$ be specified by the inequalities%
\begin{equation}
Q_{m}^{(2,2)}\left( 0\right) \leq 1-s_{2}=n^{-t_{2}}\leq
Q_{m-1}^{(2,2)}\left( 0\right) .  \label{probsurv}
\end{equation}%
In view of
\begin{equation*}
Q_{m}^{(2,2)}\left( 0\right) =1-H_{m}^{(2,2)}\left( 0\right) =\mathbf{P}%
\left( Z_{m2}>0|\mathbf{Z}_{0}=\mathbf{e}_{2}\right) \sim
\frac{2}{mVar\eta_{22}},
\end{equation*}%
it follows that  $ m\sim 2n^{t_{2}}/Var\eta_{22}. $
Using this fact, estimate (\ref{Neglig}) and the branching property%
\begin{equation*}
H_{n}^{(1,2)}\left( H_{m}^{(1,2)}\left( \mathbf{s}\right)
,H_{m}^{(2,2)}\left( s_{2}\right) \right) =H_{n+m}^{(1,2)}\left( \mathbf{s}%
\right),
\end{equation*}%
we conclude by (\ref{SurvivSingle}) that%
\begin{eqnarray*}
1-H_{n}^{(1,2)}\left( s_{1},s_{2}\right)
&\geq&1-H_{n}^{(1,2)}\left(
s_{1},H_{m}^{(2,2)} (0) \right) \\
&=&1-H_{n}^{(1,2)}\left( H_{m}^{(1,2)}(\mathbf{0})
,H_{m}^{(2,2)}( 0)\right)+O(n^{-1})\\
&=&Q_{n+m}^{(1,2)}(\mathbf{0})+O(n^{-1})= (1+o(1))\,c_{1}\left(
n+m\right) ^{-1/2}+O(n^{-1}).
\end{eqnarray*}%
Clearly, the result remains valid when $\geq$ is replaced by
$\leq$ with $m$ replaced by $m-1.$ Therefore, $
1-H_{n}^{(1,2)}\left( s_{1},s_{2}\right) \asymp n^{-1/2} $ if
$\,t_{2}\in (0,1],$ and  $ 1-H_{n}^{(1,2)}\left(
s_{1},s_{2}\right) \asymp n^{-t_{2}/2}$ if $t_{2}\in (1,2).$
 This proves the first two relationships of
the lemma.

Consider now the case $t_{2}\geq 2$. In view of (\ref{MomentSingle3})%
\begin{eqnarray*}
1-H_{n}^{(1,1)}\left( s_{1}\right) &=&1-H_{n}^{(1,2)}\left(
s_{1},1\right)
\leq 1-H_{n}^{(1,2)}\left( s_{1},s_{2}\right) \\
&\leq &1-H_{n}^{(1,1)}\left( s_{1}\right)
+n^{-t_{2}}\mathbf{E}\left[
Z_{n2}|\mathbf{Z}_{0}=\mathbf{e}_{1}\right] \\
&=&1-H_{n}^{(1,1)}\left( s_{1}\right)
+(1+o(1))\,c_{12}n^{1-t_{2}}.
\end{eqnarray*}%
Recalling that $1-s_{1}=n^{-t_{1}}$ and selecting $m=m\left(
s_{1}\right) $
similar to (\ref{probsurv}) we get%
\begin{equation}
1-H_{n}^{(1,1)}\left( s_{1}\right) \sim 1-H_{n+m}^{(1,1)}\left(
0\right) \asymp \frac{1}{n^{t_{1}}+n}.  \label{oneSide}
\end{equation}%
Hence, if $t_{1}< 1$ then  $ 1-H_{n}^{(1,2)}\left(
s_{1},s_{2}\right) \asymp n^{-1} $ as claimed.

The  statement for $t_{1}\geq1,t_{2}\geq 2$ follows from (\ref{twoSide}).
\end{proof}

\section{ \ Decomposable branching processes in random environment\label%
{SecREnvironm}}

The model of branching processes in random environment which we
are dealing with is a combination of the processes introduced by
Smith and Wilkinson \cite{SW69} and the ordinary decomposable
multitype Galton-Watson processes. To give a formal description of
the model denote by $\mathcal{M}$ the space
of probability measures on $\mathbb{N}_{0}^{N+1},$ where $\mathbb{N}%
_{0}:=\{0,1,2,...\}$ and let $\mathfrak{e}$ be a random variable
with values
in $\mathcal{M}$. An infinite sequence $\mathcal{E}=(\mathfrak{e}_{1},%
\mathfrak{e}_{2},\ldots )$ of i.i.d. copies of $\mathfrak{e}$ is
said to form a \emph{random environment}.

We associate with $\mathfrak{e}$ and $\mathfrak{e}_{n},n=1,2,...$
random vectors $\left( \xi _{0},...,\xi _{N}\right) $ and $\left(
\xi
_{0}^{(n)},...,\xi _{N}^{(n)}\right) $ such that for $\mathbf{k}\in \mathbb{N%
}_{0}^{N+1}$
\begin{equation*}
\mathbb{P}\left( \left( \xi _{0},...,\xi _{N}\right) =\mathbf{k}|\mathfrak{e}%
\right) =\mathfrak{e}\left( \left\{ \mathbf{k}\right\} \right) ,\ \mathbb{P}%
\left( \left( \xi _{0}^{(n)},...,\xi _{N}^{(n)}\right) =\mathbf{k}|\mathfrak{%
e}_{n}\right) =\mathfrak{e}_{n}\left( \left\{ \mathbf{k}\right\}
\right) .
\end{equation*}

We now specify a branching process $\left(
X_{n},\mathbf{Z}_{n}\right) =\left( X_{n},Z_{n1},...,Z_{nN}\right)
$ in random environment $\mathcal{E}$ with types $0,1,...,N$ as
follows.

1) $\left( X_{0},\mathbf{Z}_{0}\right) =\left(
1,\mathbf{0}\right).$

2) Given $\mathcal{E}\mathbf{=}\left( e_{1},e_{2},...\right) $ and
$\left(
X_{n-1},\mathbf{Z}_{n-1}\right) ,n\geq 1$%
\begin{equation*}
X_{n}=\sum_{k=1}^{X_{n-1}}\xi
_{k0}^{(n-1)},\,Z_{nj}=\sum_{k=1}^{X_{n-1}}\xi
_{kj}^{(n-1)}+\sum_{i=1}^{j}\sum_{k=1}^{Z_{\left( n-1\right)
i}}\eta _{k,ij}^{(n-1)},\quad j=1,...,N
\end{equation*}%
where the tuples $\left( \xi _{k0}^{(n-1)},\xi
_{k1}^{(n-1)},\ldots ,\xi _{kN}^{(n-1)}\right)
,\,k=1,2,...,X_{n-1}$ are i.i.d. random vectors with distribution
$e_{n-1}$ i.e., given $\mathfrak{e}_{n-1}=e_{n-1}$ distributed as
$\left( \xi _{0}^{(n-1)},\xi _{1}^{(n-1)},\ldots ,\xi
_{N}^{(n-1)}\right) ,$ and the tuples $\left( \eta
_{kii}^{(n-1)},\eta _{ki,i+1}^{(n-1)},\ldots
,\eta _{kiN}^{(n-1)}\right) $ are independent random vectors distributed as $%
\left( \eta _{ii},\eta _{i,i+1},\ldots ,\eta _{iN}\right) $ for
$i=1,2,...N,$
i.e., in accordance with the respective probability generating function $%
h^{(i,N)}(\mathbf{s})$ in (\ref{DefNONimmigr}).

Informally, $\xi _{kj}^{(n-1)}$ is the number of type $j$ children
produced by the $k$-th particle of type $0$ of generation $n-1$,
while $\eta _{k,ij}^{(n-1)}$ is the number of type $j$ children
produced by the $k$-th particle of type $i$ of generation $n-1$.

We denote by $\mathbb{P}$ and $\mathbb{E}$ the corresponding
probability measure and expectation on the underlying probability
space to distinguish them from the probability measure and
expectation in constant environment specified by the symbols
$\mathbf{P}$ and $\mathbf{E}$.

Thus, in our model particles of type $0$ belonging to the
$(n-1)$-th generation give birth in total to $X_{n}$ particles of
their own type and to the tuple $\mathbf{Y}_{n}=\left(
Y_{n1},...,Y_{nN}\right) $ of daughter
particles of types $1,2,...,N,$ where%
\begin{equation}
Y_{nj}=\sum_{k=1}^{X_{n-1}}\xi _{kj}^{(n-1)}.  \label{DDefY}
\end{equation}%
In particular,
$
\mathbf{Y}_{1}=\left( Y_{11},...,Y_{1N}\right) =\left( \xi
_{1}^{(0)},...,\xi _{N}^{(0)}\right) =\mathbf{Z}_{1}.
$

Finally, each particle of type $i=1,2,...,N$ generates its own
(decomposable, if $i<N$) process with $N-i+1$ types evolving in a
constant environment.

Let $
\mu _{1}=\mathbb{E}\left[ \xi _{0}|\mathfrak{e}\right] ,\quad \mu _{2}=%
\mathbb{E}\left[ \xi _{0}\left( \xi _{0}-1\right)
|\mathfrak{e}\right] ,
$
and%
\begin{equation*}
\theta _{i}=\mathbb{E}\left[ \xi _{i}|\mathfrak{e}\right] ,\quad
i=1,2,...,N,\quad \Theta _{1}:=\sum_{l=1}^{N}\theta _{l}.
\end{equation*}

Our assumptions on the characteristics of the process we consider
are formulated as

\textbf{Hypothesis A:}
\begin{itemize}
\item The initial state of the process is $\left( X_{0},\mathbf{Z}%
_{0}\right) =\left( 1,\mathbf{0}\right) ;$

\item particles of type $0$ form (on their own) a critical \textbf{\ }%
branching process in a random environment, such that
\begin{equation}
\mathbb{E}\log \mu _{1}=0,\,\mathbb{E}\log ^{2}\mu _{1}\in
(0,\infty ); \label{RandCrit}
\end{equation}

\item particles of type $0$ produce particles of type $1$ with a
positive probability and
\begin{equation*}
\mathbb{P}\left( \theta _{1}>0\right) =1;
\end{equation*}
\item particles of each type form (on their own) critical
branching processes which are independent of the environment, i.e.
$ m_{ii}=\mathbf{E}\eta _{ii}=1,\,i=1,2,...,N;$ \item particles of
type $i=1,2,...,N-1\,$ produce particles of type $i+1$ with a
positive probability, i.e.,\, $ m_{i,i+1}=\mathbf{E}\eta
_{i,i+1}>0,\,i=1,2,...,N-1; $

\item The second moments of the offspring numbers are finite%
\begin{equation*}
\mathbf{E}\eta _{ij}^{2}<\infty ,1\leq i\leq j\leq N\quad
\mbox{with}\quad b_{i}=\frac{1}{2}Var\,\eta _{ii}\in \left(
0,\infty \right) .
\end{equation*}
\end{itemize}

The following theorem is the main result of the paper:

\begin{theorem}
\label{T_main}If Hypothesis A is valid and
\begin{equation}
\mathbb{E}\left[ \mu _{1}^{-1}\right] <\infty ,\ \mathbb{E}\left[
\mu _{2}\mu _{1}^{-2}\left( 1+\max \left( 0,\log \mu _{1}\right)
\right) \right] <\infty ,  \label{Aff}
\end{equation}%
\textbf{\ }then there exists a positive constant $K_{0}$ such that
\begin{equation}
\mathbb{P}\left( \mathbf{Z}_{n}\neq 0|X_{0}=1,\mathbf{Z}_{0}=\mathbf{0}%
\right) \sim \frac{2^{N-1}K_{0}}{\log n}  \label{As_survivN}
\end{equation}%
and for any positive $t_{1},t_{2},...,t_{N}$
\begin{eqnarray}
&&\lim_{n\rightarrow \infty }\mathbb{P}\left( \frac{\log
Z_{ni}}{\log n}\leq t_{i},i=1,...,N\,|\,Z_{n1}>0\right) =G\left(
t_{1},...,t_{N}\right)  \notag
\\
&&\qquad\qquad\qquad\qquad\qquad =1-\frac{1}{1+\max (0,\min_{1\leq
l\leq N}\left( t_{l}-l\right) )}.  \label{LimSurvivN}
\end{eqnarray}
\end{theorem}

The proof of the theorem is divided into several stages.

Let%
\begin{equation*}
\mathrm{T}=\min \{n\geq 0:X_{n}=0\}.
\end{equation*}%
According to \cite[Theorem 1]{GK00}, if conditions (\ref{RandCrit}) and (\ref%
{Aff}) are valid then for a positive constant $c$
\begin{equation}
\mathbb{P}\left( X_{n}>0\right) =\mathbb{P}\left(
\mathrm{T}>n\right) \sim \frac{c}{\sqrt{n}},\ n\rightarrow \infty
.  \label{koz}
\end{equation}

Set $S_{n}:=\sum_{k=0}^{n-1}X_{k}$ and $A_{n}=\max_{0\leq k\leq
n-1}X_{k}$, so that $S_{\mathrm{T}}$ and $~A_{\mathrm{T}}$ give
the total number ever born of type $0$ particles and the maximal
generation size of type $0$ particles.

\begin{lemma}
\label{L_AfanTot}(see \cite{Af99}) If conditions (\ref{RandCrit}) and (\ref%
{Aff}) are valid then there exists a constant $K_{0}\in \left(
0,\infty \right) $ such that
\begin{equation}
\mathbb{P}\left( S_{\mathrm{T}}>x\right) \sim \mathbb{P}\left( A_{\mathrm{T}%
}>x\right) \sim \frac{K_{0}}{\log x},\ x\rightarrow \infty .
\label{As_Total2}
\end{equation}
\end{lemma}

In fact, the representation (\ref{As_Total2}) has been proved in
\cite{Af99} under conditions (\ref{Aff}) and (\ref{RandCrit}) only
for the case when the probability generating functions
$f_{n}\left( s,\mathbf{1}^{(N)}\right) $ are linear-fractional
with probability 1. However, this restriction is easily removed
using the results established later on for the general case in
\cite{GK00} and \cite{AGKV}.

Let now
$
\left\Vert \mathbf{Y}_{n}\right\Vert =Y_{n1}+...+Y_{nN},~\zeta
_{k}^{(n)}=\xi _{k1}^{(n-1)}+\ldots +\xi _{kN}^{(n-1)}
$
and%
\begin{eqnarray*}
L_{nj}
&=&\sum_{l=1}^{n}Y_{lj}=\sum_{l=1}^{n}\sum_{k=1}^{X_{l-1}}\xi
_{kj}^{(l-1)},\quad B_{nj}=\max_{1\leq l\leq n}Y_{lj}, \\
L_{n} &=&\sum_{l=1}^{n}\left\Vert \mathbf{Y}_{l}\right\Vert
=\sum_{l=1}^{n}\sum_{k=1}^{X_{l-1}}\zeta _{k}^{(l-1)},\quad
B_{n}=\max_{1\leq l\leq n}\left\Vert \mathbf{Y}_{l}\right\Vert .
\end{eqnarray*}%
In particular, $L_{\mathrm{T}}$ gives the total number of daughter
particles of types $1,...,N$ produced by type $0$ particles during
the evolution of the process.

\begin{lemma}
\label{L_ytotal}If conditions (\ref{RandCrit}) and (\ref{Aff}) are
valid and $\mathbb{P}\left( \Theta _{1}>0\right) =1,$ then
\begin{equation}
\mathbb{P}\left( B_{\mathrm{T}}>x\right) \sim \mathbb{P}\left( L_{\mathrm{T}%
}>x\right) \sim \frac{K_{0}}{\log x},\ x\rightarrow \infty .
\label{AsTotY}
\end{equation}%
If conditions (\ref{Aff}), (\ref{RandCrit}) are valid and
$\mathbb{P}\left( \theta _{j}>0\right) =1$ for some $j\in \left\{
1,...,N\right\} $ then
\begin{equation}
\mathbb{P}\left( B_{\mathrm{T}j}>x\right) \sim \mathbb{P}\left( L_{\mathrm{T}%
j}>x\right) \sim \frac{K_{0}}{\log x},\ x\rightarrow \infty .
\label{AsTotYj}
\end{equation}
\end{lemma}

\begin{proof} For any $\varepsilon \in \left( 0,1\right) $ we have%
\begin{equation*}
\mathbb{P}\left( A_{\mathrm{T}}>x\right) \leq \mathbb{P}\left( B_{\mathrm{T}%
}>x^{1-\varepsilon }\right) +\mathbb{P}\left( A_{\mathrm{T}}>x;B_{\mathrm{T}%
}\leq x^{1-\varepsilon }\right) .
\end{equation*}%
Let $T_{x}=\min \left\{ k:X_{k}>x\right\} $. Then%
\begin{eqnarray*}
\mathbb{P}\left( A_{\mathrm{T}}>x;B_{\mathrm{T}}\leq
x^{1-\varepsilon }\right) &\leq &\sum_{l=1}^{\infty
}\mathbb{P}\left( T_{x}=l;\left\Vert
\mathbf{Y}_{l+1}\right\Vert \leq x^{1-\varepsilon }\right) \\
&=&\sum_{l=1}^{\infty }\mathbb{P}\left(
T_{x}=l;\sum_{k=1}^{X_{l}}\zeta
_{k}^{(l)}\leq x^{1-\varepsilon }\right) \\
&\leq &\mathbb{P}\left( A_{\mathrm{T}}>x\right) \mathbb{P}\left( \sum_{k=1}^{%
\left[ x\right] }\zeta _{k}^{(0)}\leq x^{1-\varepsilon }\right) .
\end{eqnarray*}%
Since $\mathbb{P}\left( \Theta _{1}>0\right) =1$ and $\Theta _{1}=\mathbb{E}%
\left[ \zeta _{k}^{(0)}|\mathfrak{e}\right] ,k=1,2,...,$ the law
of large numbers gives
\begin{equation*}
\lim_{x\rightarrow \infty }\mathbb{P}\left( \frac{1}{x\Theta _{1}}%
\sum_{k=1}^{\left[ x\right] }\zeta _{k}^{(0)}\leq
\frac{1}{x^{\varepsilon }\Theta _{1}}\Big|\mathfrak{e}\right)
=0\text{ \ }\mathbb{P}\text{ - a.s..}
\end{equation*}%
Thus%
\begin{equation*}
\lim \sup_{x\rightarrow \infty }\mathbb{P}\left(
\sum_{k=1}^{\left[ x\right] }\zeta _{k}^{(0)}\leq x^{1-\varepsilon
}\right) \leq \mathbb{E}\left[ \lim \sup_{x\rightarrow \infty
}\mathbb{P}\left( \sum_{k=1}^{\left[ x\right] }\zeta
_{k}^{(0)}\leq x^{1-\varepsilon }\Big|\mathfrak{e}\right) \right]
=0.
\end{equation*}%
As a result, for any $\delta >0$ and all $x\geq x_{0}(\delta )$ we
get
\begin{equation}
\left( 1-\delta \right) \mathbb{P}\left( A_{\mathrm{T}}>x\right)
\leq \mathbb{P}\left( B_{\mathrm{T}}>x^{1-\varepsilon }\right) .
\label{Upj}
\end{equation}%
To deduce for $\mathbb{P}\left( B_{\mathrm{T}}>x\right) $ an
estimate from
above we write%
\begin{equation}
\mathbb{P}\left( B_{\mathrm{T}}>x\right) \leq \mathbb{P}\left( A_{\mathrm{T}%
}>x^{1-\varepsilon }\right) +\mathbb{P}\left( B_{\mathrm{T}}>x;A_{\mathrm{T}%
}\leq x^{1-\varepsilon }\right) .  \label{Down1}
\end{equation}%
Further, letting $\mathrm{\hat{T}}_{x}=\min \left\{ k:\left\Vert \mathbf{Y}%
_{k}\right\Vert >x\right\} $ we have
\begin{eqnarray*}
\mathbb{P}\left( B_{\mathrm{T}}>x;A_{\mathrm{T}}\leq
x^{1-\varepsilon
}\right) &\leq &\mathbb{P}\left( T>x^{\varepsilon /2}\right) \\
&&+\sum_{1\leq l\leq x^{\varepsilon /2}}\mathbb{P}\left( \mathrm{\hat{T}}%
_{x}=l;A_{\mathrm{T}}\leq x^{1-\varepsilon }\right) .
\end{eqnarray*}%
By Markov inequality we see that
\begin{eqnarray*}
&&\sum_{1\leq l\leq x^{\varepsilon /2}}\mathbb{P}\left( \mathrm{\hat{T}}%
_{x}=l;A_{\mathrm{T}}\leq x^{1-\varepsilon }\right) \leq
\sum_{1\leq l\leq x^{\varepsilon /2}}\mathbb{P}\left( X_{l-1}\leq
x^{1-\varepsilon
};\left\Vert \mathbf{Y}_{l}\right\Vert >x\right) \\
&&\qquad \qquad \qquad \qquad \quad \leq x^{\varepsilon
/2}\mathbb{P}\left( \sum_{k=1}^{\left[ x^{1-\varepsilon }\right]
}\zeta _{k}^{(0)}>x\right) \leq
x^{-\varepsilon /2}\mathbb{E}\left[ \left\Vert \mathbf{Y}_{1}\right\Vert %
\right] .
\end{eqnarray*}%
Hence, recalling (\ref{koz}) we obtain  $\mathbb{P}\left(
B_{\mathrm{T}}>x;A_{\mathrm{T}}\leq x^{1-\varepsilon }\right)
=O\left( x^{-\varepsilon /4}\right) $ implying in view of
(\ref{Down1})
\begin{equation}
\mathbb{P}\left( B_{\mathrm{T}}>x\right) \leq \mathbb{P}\left( A_{\mathrm{T}%
}>x^{1-\varepsilon }\right) +O\left( x^{-\varepsilon /4}\right) .
\label{Down2}
\end{equation}%
Combining (\ref{Upj}) and (\ref{Down2}) and letting first
$x\rightarrow
\infty $ and then $\varepsilon \rightarrow 0$ justify by Lemma \ref%
{L_AfanTot} the equivalence
\begin{equation*}
\mathbb{P}\left( B_{\mathrm{T}}>x\right) \sim \mathbb{P}\left( A_{\mathrm{T}%
}>x\right) \sim \frac{K_{0}}{\log x}.
\end{equation*}%
Finally,%
$$
\mathbb{P}\left( B_{\mathrm{T}}>x\right) \leq\mathbb{P}\left( L_{\mathrm{T%
}}>x\right) \leq \mathbb{P}\left( \mathrm{T}B_{\mathrm{T}}>x\right) \\
\leq\mathbb{P}\left( B_{\mathrm{T}}>x^{1-\varepsilon }\right) +\mathbb{P}%
\left( \mathrm{T}>x^{\varepsilon }\right),
$$
and applying (\ref{koz}) and Lemma \ref{L_AfanTot} proves the
first equivalence in (\ref{AsTotY}).

One may check (\ref{AsTotYj}) by similar arguments.
\end{proof}

\begin{corollary}
\label{C_asExtTot}If conditions (\ref{RandCrit}) and (\ref{Aff})
are valid and $\mathbb{P}\left( \theta _{1}>0\right) =1,$ then, as
$n\rightarrow \infty
$%
\begin{equation*}
F(n):=\mathbb{E}\left[ 1-\exp \left\{ -\sum_{i=1}^{N}L_{\mathrm{T}%
i}Q_{n}^{(i,N)}(\mathbf{0})\right\} \right] \sim
\frac{2^{N-1}K_{0}}{\log n}.
\end{equation*}
\end{corollary}

\begin{proof} Clearly,%
\begin{equation*}
L_{\mathrm{T}1}Q_{n}^{(1,N)}(\mathbf{0})\leq \sum_{i=1}^{N}L_{\mathrm{T}%
i}Q_{n}^{(i,N)}(\mathbf{0})\leq L_{\mathrm{T}}\sum_{i=1}^{N}Q_{n}^{(i,N)}(%
\mathbf{0})
\end{equation*}%
and, by (\ref{SurvivSingle})%
\begin{equation*}
\sum_{i=1}^{N}Q_{n}^{(i,N)}(\mathbf{0})\sim
Q_{n}^{(1,N)}(\mathbf{0})\sim c_{1}n^{-1/2^{(N-1)}}.
\end{equation*}%
To finish the proof of the corollary it remains to observe that
\begin{equation}
\mathbb{E}\left[ 1-e^{-\lambda L_{\mathrm{T}}}\right] \sim
\mathbb{E}\left[ 1-e^{-\lambda L_{\mathrm{T}1}}\right] \sim
\frac{K_{0}}{\log (1/\lambda )},\ \lambda \rightarrow +0,
\label{Lapla}
\end{equation}%
due to Lemma \ref{L_ytotal} and the Tauberian theorem \cite[Ch.
XIII.5, Theorem 4]{Fe2} applied, for instance, to the right hand
side of
\begin{equation*}
\lambda ^{-1}\mathbf{E}\left[ 1-e^{-\lambda L_{\mathrm{T}}}\right]
=\int_{0}^{\infty }\mathbf{P}\left( L_{\mathrm{T}}>x\right)
e^{-\lambda x}dx,
\end{equation*}%
and to use the inequalities%
\begin{equation*}
\mathbb{E}\left[ 1-\exp \left\{ -L_{\mathrm{T}1}Q_{n}^{(1,N)}(\mathbf{0}%
)\right\} \right] \leq F(n)\leq \mathbb{E}\left[ 1-\exp \left\{ -L_{\mathrm{T%
}}\sum_{i=1}^{N}Q_{n}^{(i,N)}(\mathbf{0})\right\} \right] .
\end{equation*}
\end{proof}
\begin{proof} of Theorem \ref{T_main}. We first check (\ref%
{As_survivN}). Notice that each particle of type $i$ of generation
$n$ has either a mother of type $0$ (of generation $n-1$), or an
ancestor of generation $k, 1\leq k <n$ whose mother is of type
$0;$\, recall that the number of particles of type $i$ of
generation $k$ having a mother of type $0$ is denoted by $Y_{ki}.$
By a decomposition of $Z_{ni}$ based on this fact and using the
branching property, we get:
\begin{equation}
\mathbb{E}\left[ 1-s_{1}^{Z_{n1}}...s_{N}^{Z_{nN}}\right]
=\mathbb{E}\left[
1-\prod\limits_{k=1}^{n}\prod\limits_{i=1}^{N}\left( H_{n-k}^{(i,N)}(%
\mathbf{s})\right) ^{Y_{ki}}\right]  \notag \\
=\mathbb{E}\left[ 1-e^{R(n;\mathbf{s})}\right], \label{BasicEqv}
\end{equation}%
where $H_0^{(i,N)}(\mathbf{s})= s_i$ by convention, and
\begin{equation*}
R(n;\mathbf{s})=\sum_{k=1}^{n}\sum_{i=1}^{N}Y_{ki}\log H_{n-k}^{(i,N)}(%
\mathbf{s}).
\end{equation*}%
In particular,
\begin{equation*}
\mathbb{P}\left( \mathbf{Z}_{n}\neq \mathbf{0}\right)
=\mathbb{E}\left[
1-e^{R(n;\mathbf{0})};\mathrm{T}\leq \sqrt{n}\right] +O\left( \mathbb{P}%
\left( \mathrm{T}>\sqrt{n}\right) \right) .
\end{equation*}%
Since $\log (1-x)\sim -x$ as $\ x\rightarrow +0$ and for $k\leq
\sqrt{n}$ and $n\rightarrow \infty $
\begin{equation*}
Q_{n}^{(i,N)}(\mathbf{0})=1-H_{n}^{(i,N)}(\mathbf{0})\leq Q_{n-k}^{(i,N)}(%
\mathbf{0})\leq Q_{n-\sqrt{n}}^{(i,N)}(\mathbf{0})=\left(
1+o(1)\right) Q_{n}^{(i,N)}(\mathbf{0}),
\end{equation*}%
we obtain%
\begin{eqnarray*}
&&\mathbb{E}\left[ e^{R(n;\mathbf{0})};\mathrm{T}\leq \sqrt{n}\right] =%
\mathbb{E}\left[ \exp \left\{ -\left( 1+o(1)\right)
\sum_{i=1}^{N}L_{ni}Q_{n}^{(i,N)}(\mathbf{0})\right\} ;\mathrm{T}\leq \sqrt{n%
}\right] \\
&&\qquad\qquad=\mathbb{E}\left[ \exp \left\{ -\left( 1+o(1)\right)
\sum_{i=1}^{N}L_{\mathrm{T}i}Q_{n}^{(i,N)}(\mathbf{0})\right\} ;\mathrm{T}%
\leq \sqrt{n}\right] \\
&&\qquad\qquad=\mathbb{E}\left[ \exp \left\{ -\left( 1+o(1)\right)
\sum_{i=1}^{N}L_{\mathrm{T}i}Q_{n}^{(i,N)}(\mathbf{0})\right\}
\right] -O\left( \mathbb{P}\left( \mathrm{T}>\sqrt{n}\right)
\right) .
\end{eqnarray*}%
Thus,%
\begin{equation}
\mathbb{P}\left( \mathbf{Z}_{n}\neq \mathbf{0}\right)
=\mathbb{E}\left[
1-\exp \left\{ -(1+o(1))\sum_{i=1}^{N}L_{\mathrm{T}i}Q_{n}^{(i,N)}(\mathbf{0}%
)\right\} \right] +O\left( \mathbb{P}\left(
\mathrm{T}>\sqrt{n}\right) \right),  \label{TrueSurv}
\end{equation}%
and (\ref{As_survivN}) follows from Corollary \ref{C_asExtTot} and (\ref{koz}%
).

Now we prove (\ref{LimSurvivN}). Recall that we always take $X_{0}=1,\mathbf{%
Z}_{0}=\mathbf{0}$.

Consider first the case $N=1$. Writing for simplicity $%
Y_{k}=Y_{k1},Z_{n}=Z_{n1}$, $s=s_{1}$ and $H_{n}(s)=H_{n}^{(1,1)}(s)=\mathbf{%
E}\left[ s^{Z_{n}}|Z_{0}=1\right] $ we have%
\begin{equation*}
\mathbb{E}\left[ s^{Z_{n}}|Z_{n}>0\right] =\frac{\mathbb{E}\left[ s^{Z_{n}}%
\right] -\mathbb{E}\left( Z_{n}=0\right) }{\mathbb{P}\left( Z_{n}>0\right) }%
=1-\frac{\mathbb{E}\left[ 1-s^{Z_{n}}\right] }{\mathbb{P}\left(
Z_{n}>0\right)},
\end{equation*}%
and by (\ref{BasicEqv})%
\begin{equation*}
\mathbb{E}\left[ 1-s^{Z_{n}}\right] =\mathbb{E}\left[ 1-\exp
\left\{ \sum_{k=1}^{n}Y_{k}\log H_{n-k}(s)\right\} \right] .
\end{equation*}%
By the criticality condition $ 1-H_{n}(0)\sim (b_{1}n)^{-1}. $
Thus, if $s=e^{-\lambda /\left( b_{1}n^{t}\right)},$ then
\begin{equation*}
1-s\sim \lambda /\left( b_{1}n^{t}\right) \sim 1-H_{\left[ n^{t}/\lambda %
\right] }(0),
\end{equation*}%
where $[x]$ denotes the integral part of $x$. Hence it follows that for any $t>1$ as $n\rightarrow \infty $%
\begin{equation*}
1-H_{n}\left( e^{\lambda /n^{t}}\right) \sim 1-H_{n}\left(
H_{\left[ n^{t}/\lambda \right] }(0)\right) =1-H_{n+\left[
n^{t}/\lambda \right] }\left( 0\right) \sim \lambda /\left(
b_{1}n^{t}\right) .
\end{equation*}%
This, similar to the previous estimates for the survival probability of the $%
(N+1)$-type branching process gives (recall that $\left(
X_{0},Z_{0}\right) =(1,0)$)
\begin{equation*}
\mathbb{E}\left[ 1-\exp \left\{ -\lambda Z_{n}/\left(
b_{1}n^{t}\right)
\right\} \right] \sim \mathbb{E}\left[ 1-\exp \left\{ -\lambda cn^{-t}L_{%
\mathrm{T}1}\right\} \right] \sim \frac{K_{0}}{t\log n}.
\end{equation*}%
Since $\mathbf{P}(Z_n>0)\sim K_0/ \log n,$ it follows that for any
fixed $t>1$ and $\lambda >0$
\begin{equation*}
\lim_{n\rightarrow \infty }\mathbb{E}\left[ \exp \left\{ -\lambda
Z_{n}/\left( b_{1}n^{t}\right) \right\} |Z_{n}>0\right]
=1-\frac{1}{t}.
\end{equation*}%
This implies that the conditional law of $Z_n/(b_1 n^t)$ given
$Z_n>0$ converges to the law of a random variable $X$ with
$\mathbf{P}(X=0)= 1-t^{-1}$
and $\mathbf{P}(X= +\infty) = t^{-1}.$ Therefore, for any $t>1$%
\begin{eqnarray}
G(t) &=&\lim_{n\rightarrow \infty }\mathbb{P}\left(
n^{-t}Z_{n}\leq
b_{1}|Z_{n}>0\right)  \notag \\
&=&\lim_{n\rightarrow \infty }\mathbb{P}\left( \frac{\log
Z_{n}}{\log n}\leq t\,\Big|\,Z_{n}>0\right) =1-\frac{1}{t}.
\label{Lott}
\end{eqnarray}%
Since $\lim_{t\downarrow 1}G(t)=0$ we may rewrite (\ref{Lott}) for
any $t>0$ as
\begin{equation}
\lim_{n\rightarrow \infty }\mathbb{P}\left( \frac{\log Z_{n}}{\log n}\leq t%
\Big|Z_{n}>0\right) =1-\frac{1}{1+\max \left( 0,t-1\right)},
\label{Lott1}
\end{equation}%
as desired.

Now we consider the case $N\geq 2$ and use the equality
\begin{eqnarray}
\mathbb{E}\left[ s_{1}^{Z_{n1}}...s_{N}^{Z_{nN}}|Z_{n1}>0\right] &=&\frac{%
\mathbb{E}\left[ 1-s_{2}^{Z_{n2}}...s_{N}^{Z_{nN}}I\left\{ Z_{n1}=0\right\} %
\right] }{\mathbb{P}\left( Z_{n1}>0\right) }  \notag \\
&&-\frac{\mathbb{E}\left[ 1-s_{1}^{Z_{n1}}...s_{N}^{Z_{nN}}\right] }{\mathbb{%
P}\left( Z_{n1}>0\right) }.  \label{EqDiff}
\end{eqnarray}

We study each term at the right-hand side of (\ref{EqDiff}) separately. By (%
\ref{BasicEqv}) and $\log (1-x)\sim -x,x\rightarrow +0$ we see that, as $%
n\rightarrow \infty $
\begin{equation}
\mathbb{E}\left[ 1-s_{1}^{Z_{n1}}...s_{N}^{Z_{nN}}\right]
=\mathbb{E}\left[ 1-\exp \left\{
-(1+o(1))R_{N}(n,\mathbf{s})\right\} \right],   \label{DefR}
\end{equation}%
where%
\begin{equation*}
R_{N}(n,\mathbf{s}):=\sum_{k=1}^{n}\sum_{i=1}^{N}Y_{ki}Q_{n-k}^{(i,N)}(%
\mathbf{s}).
\end{equation*}%
Let now $t_{1},...t_{N}$ be a tuple of positive numbers satisfying (\ref%
{RanT}). It follows from Lemma~\ref{L_multit} that, for $%
1-s_{l}=n^{-t_{l}},l=1,...,N$%
\begin{equation}
Q_{n}^{(i,N)}(\mathbf{s})\asymp n^{-\min_{i\leq l\leq N}\left(
t_{l}-l+i\right) }=n^{-i-\min_{i\leq l\leq N}\left(
t_{l}-l\right). } \label{Ba1}
\end{equation}%
Since
\begin{equation}
\min_{1\leq i\leq N}\min_{i\leq l\leq N}\left( t_{l}-l+i\right)
=\min_{1\leq l\leq N}\left( t_{l}-l+1\right)\geq 1  \label{tutu}
\end{equation}%
by our conditions, we have as $n\rightarrow \infty $:%
\begin{equation*}
Q_{n}^{(i,N)}(\mathbf{s})\ll Q_{n}^{(1,N)}(\mathbf{s})\asymp
n^{-\min_{1\leq l\leq N}\left( t_{l}-l+1\right) }.
\end{equation*}%
Thus, there exist constants $C_{j},j=1,2,3,4$ such that, on the set $\mathrm{%
T}\leq \sqrt{n}$ the estimates
\begin{equation*}
C_{1}L_{\mathrm{T}1}Q_{n}^{(1,N)}(\mathbf{s})\leq
R_{N}(n,\mathbf{s})\leq
\sum_{k=1}^{n}\sum_{i=1}^{N}Y_{ki}Q_{n-k}^{(i,N)}(\mathbf{s})\leq C_{2}L_{%
\mathrm{T}}\sum_{i=1}^{N}Q_{n}^{(i,N)}(\mathbf{s})
\end{equation*}%
are valid for all sufficiently large $n$. This, in turn, implies
\begin{equation}
C_{3}L_{\mathrm{T}1}n^{-\min_{1\leq l\leq N}\left(
t_{l}-l+1\right) }\leq R_{N}(n,\mathbf{s})\leq
C_{4}n^{-\min_{1\leq l\leq N}\left( t_{l}-l+1\right)
}L_{\mathrm{T}}.  \label{EstR_1}
\end{equation}%
Using the estimates above and (\ref{Lapla}) we get for the selected $%
t_{1},...,t_{N}$, as $n\rightarrow \infty $%
\begin{eqnarray*}
\mathbb{E}\left[ 1-\exp \left\{ -R_{N}(n,\mathbf{s})\right\}
;\mathrm{T}\leq
\sqrt{n}\right]  &=&\frac{1}{\log n}\frac{\left( 1+o(1)\right) K_{0}}{%
1+\min_{1\leq l\leq N}\left( t_{l}-l\right) } \\
&& \\
&&\qquad \qquad +\,O\left( \mathbb{P}\left(
\mathrm{T}>\sqrt{n}\right) \right) ,
\end{eqnarray*}%
which leads on account of (\ref{koz}) to
\begin{equation}
\lim_{n\rightarrow \infty }\left( \log n\right) \mathbb{E}\left[
1-s_{1}^{Z_{n1}}...s_{N}^{Z_{nN}}\right]
=\frac{K_{0}}{1+\min_{1\leq l\leq N}\left( t_{l}-l\right) }.
\label{Dranget}
\end{equation}%
Thus,%
\begin{equation*}
\lim_{n\rightarrow \infty }\frac{\mathbb{E}\left[
1-s_{1}^{Z_{n1}}...s_{N}^{Z_{nN}}\right] }{\mathbb{P}\left( Z_{n1}>0\right) }%
=\frac{1}{1+\min_{1\leq l\leq N}\left( t_{l}-l\right) }<1.
\end{equation*}%
Further,
\begin{equation*}
\mathbb{E}\left[ 1-s_{2}^{Z_{n2}}...s_{N}^{Z_{nN}}I\left\{ Z_{n1}=0\right\} %
\right] =\mathbb{E}\left[ 1-\exp \left\{
\sum_{k=1}^{n}\sum_{i=1}^{N}Y_{ki}\log H_{n-k}^{(i,N)}(0,\mathbf{s}%
_{2})\right\} \right] .
\end{equation*}%
By definitions of $H_{n}^{(i,N)}(\mathbf{s})$, estimates
(\ref{Ba1}) and the choice of $s_{i},i=2,...,N$ we have
\begin{equation*}
1-H_{n}^{(i,N)}(0,\mathbf{s}_{2})=1-H_{n}^{(i,N)}(\mathbf{s})=Q_{n}^{(i,N)}(%
\mathbf{s})\asymp n^{-\min_{i\leq l\leq N}\left( t_{l}-l+i\right)
}=o\left( n^{-1}\right) \text{.}
\end{equation*}%
Besides, as $n\rightarrow \infty $%
\begin{equation}
1-H_{n}^{(1,N)}(0,\mathbf{s}_{2})=Q_{n}^{(1,N)}(0,\mathbf{s}_{2})\sim
c_{1}n^{-1}  \label{AAA}
\end{equation}%
by Lemma \ref{L_singleterm}. Hence it follows that on the set
$T\leq \sqrt{n},$
\begin{eqnarray*}
\sum_{k=0}^{T-1}\sum_{i=1}^{N}Y_{ki}\log
H_{n-k}^{(i,N)}(0,\mathbf{s}_{2}) &=&-\left( 1+o(1)\right)
\sum_{k=0}^{T-1}\sum_{i=1}^{N}Y_{ki}Q_{n-k}^{(i,N)}(0,\mathbf{s}_{2}) \\
&=&-\left( 1+o(1)\right) \sum_{i=1}^{N}L_{\mathrm{T}i}Q_{n}^{(i,N)}(0,%
\mathbf{s}_{2})
\end{eqnarray*}%
and, moreover,%
\begin{equation*}
Q_{n}^{(1,N)}(0,\mathbf{s}_{2})L_{\mathrm{T}1}\leq \sum_{i=1}^{N}L_{\mathrm{T%
}i}Q_{n}^{(i,N)}(0,\mathbf{s}_{2})\leq C_{2}Q_{n}^{(1,N)}(0,\mathbf{s}%
_{2})L_{\mathrm{T}}.
\end{equation*}%
Using now the same line of arguments as earlier one may show that
\begin{equation*}
\lim_{n\rightarrow \infty }\mathbb{E}\left[
1-s_{2}^{Z_{n2}}...s_{N}^{Z_{nN}}I\left\{ Z_{n1}=0\right\} \right]
\log n=K_{0},
\end{equation*}%
implying by (\ref{As_survivN}) with $N=1$ that
\begin{equation*}
\lim_{n\rightarrow \infty }\frac{\mathbb{E}\left[
1-s_{2}^{Z_{n2}}...s_{N}^{Z_{nN}}I\left\{ Z_{n1}=0\right\} \right] }{\mathbb{%
P}\left( Z_{n1}>0\right) }=1.
\end{equation*}%
As a result, given (\ref{RanT}) we have
\begin{equation*}
G(t_{1},...,t_{N})=\lim_{n\rightarrow \infty }\mathbb{E}\left[
s_{1}^{Z_{n1}}...s_{N}^{Z_{nN}}|Z_{n1}>0\right]
=1-\frac{1}{1+\min_{1\leq l\leq N}\left( t_{l}-l\right) }.
\end{equation*}%
Since
$
\lim_{\min_{1\leq l\leq N}\left( t_{l}-l\right) \downarrow
0}G(t_{1},...,t_{N})=0
$
we conclude by the same arguments that have been used to derive
(\ref{Lott}) and (\ref{Lott1}) that
\begin{equation*}
\lim_{n\rightarrow \infty }\mathbb{E}\left[
s_{1}^{Z_{n1}}...s_{N}^{Z_{nN}}|Z_{n1}>0\right] =1-\frac{1}{1+\max
(0,\min_{1\leq l\leq N}\left( t_{l}-l\right) )}
\end{equation*}%
for all positive $t_{1},...,t_{N},$ completing the proof of
Theorem \ref{T_main}. \end{proof}

\section{The case of three types\label{Secthree}}

It follows from (\ref{SurvivSingle}) that for a strongly critical
$N$-type decomposable branching process in a constant environment
\begin{equation*}
\mathbf{P}\left( \mathbf{Z}_{n}\neq \mathbf{0\,}|\,\mathbf{Z}_{0}=\mathbf{e}%
_{1}\right) \sim \mathbf{P}\left( Z_{n1}+...+Z_{n,N-1}=0,Z_{nN}>0\mathbf{\,}%
|\,\mathbf{Z}_{0}=\mathbf{e}_{1}\right) .
\end{equation*}%
Thus, given the condition $\left\{ \mathbf{Z}_{n}\neq
\mathbf{0}\right\} $ we observe in the limit, as $n\rightarrow
\infty $ only type $N$ particles. This is not the case for the
strongly critical $\left( N+1\right) $-type decomposable branching
process in a random environment. We justify this claim by
considering a strongly critical branching process with three types
and prove the following statement.

\begin{theorem}
\label{T_twodim}Let $N=2$. If hypothesis A is valid then%
\begin{equation}
\lim_{n\rightarrow \infty }\mathbb{P}\left( \frac{\log
Z_{n1}}{\log n}\leq
t_{1},\frac{\log Z_{n2}}{\log n}\leq t_{2}\Big|\mathbf{Z}_{n}\neq \mathbf{0},%
\text{ }X_{0}=1,\mathbf{Z}_{0}=\mathbf{0}\right) =A(t_{1},t_{2}),
\label{Distrib}
\end{equation}%
where%
\begin{equation*}
A(t_{1},t_{2})=\left\{
\begin{array}{cccc}
0, & \text{if} & t_{1}\in \lbrack 0,\infty ), & 0\leq t_{2}\leq 1; \\
1-t_{2}^{-1}, & \text{if} & t_{1}\in \lbrack 0,\infty ), & 1<t_{2}<2; \\
1/2, & \text{if} & 0\leq t_{1}<1, & t_{2}\geq 2; \\
1-\frac{1}{2}\frac{1}{1+\min \left( t_{1}-1,t_{2}-2\right) }, &
\text{if} &
t_{1}\geq 1, & t_{2}\geq 2.%
\end{array}%
\right.
\end{equation*}
\end{theorem}

\begin{remark} Since the survival probability of particles of
type $0$ up to moment $n$ is of order $n^{-1/2}$, particles of
this type are absent in the limit. \end{remark}

\begin{remark}  Since
$
\lim_{\min (t_{1},t_{2}-1)\downarrow 0}A(t_{1},t_{2})=0,
$
Theorem \ref{T_twodim} gives a complete description of the
limiting distribution for the left-hand side of (\ref{Distrib}).
\end{remark}

\begin{proof} of Theorem \ref{T_twodim}. We have
\begin{equation*}
\mathbb{E}\left[ s_{1}^{Z_{n1}}s_{2}^{Z_{n2}}|\mathbf{Z}_{n}\neq \mathbf{0}%
\right] =1-\frac{\mathbb{E}\left[ 1-s_{1}^{Z_{n1}}s_{2}^{Z_{n2}}\right] }{%
\mathbb{P}\left( \mathbf{Z}_{n}\neq \mathbf{0}\right)},
\end{equation*}%
where
\begin{equation*}
\mathbb{E}\left[ 1-s_{1}^{Z_{n1}}s_{2}^{Z_{n2}}\right]
=\mathbb{E}\left[
1-\exp \left\{ \sum_{k=1}^{n}\sum_{i=1}^{2}Y_{ki}\log H_{n-k}^{(i,N)}(%
\mathbf{s})\right\} \right].
\end{equation*}%
Let now $1-s_{i}=n^{-t_{i}}$. If $t_{1}\geq 1$ and $t_{2}\geq 2$ then by (\ref%
{As_survivN}) (with $N=2$) and (\ref{Dranget}) we have
\begin{equation*}
A(t_{1},t_{2})=1-\lim_{n\rightarrow \infty }\frac{\mathbb{E}\left[
1-s_{1}^{Z_{n1}}s_{2}^{Z_{n2}}\right] }{\mathbb{P}\left(
\mathbf{Z}_{n}\neq \mathbf{0}\right)
}=1-\frac{1}{2}\frac{1}{1+\min \left( t_{1}-1,t_{2}-2\right)},
\end{equation*}%
proving Theorem \ref{T_twodim} for $\min \left(
t_{1}-1,t_{2}-2\right)\geq 0$. Observe that
$$ \lim_{\min \left(
t_{1}-1,t_{2}-2\right) \downarrow 0}A(t_{1},t_{2})=1/2, $$
 and,
therefore, contrary to the case $\mathbb{P}\left( Z_{n1}>0\right)
$ we
need to analyze the case of positive $t_{1},t_{2}$ meeting the condition $%
\min \left( t_{1}-1,t_{2}-2\right) <0$ in more detail.

The same as in the proof of Theorem \ref{T_main}, it is necessary
to obtain estimates from above and below for
\begin{equation*}
R_{2}(n,\mathbf{s})=\sum_{k=1}^{n}\sum_{i=1}^{2}Y_{ki}Q_{n-k}^{(i,2)}(%
\mathbf{s})
\end{equation*}%
\ given $T\leq \sqrt{n}$. Observe that in view of Lemma
\ref{L_twodimen} and the representation
\begin{equation*}
Q_{n}^{(2,2)}(s_{2})=1-H_{n}^{(2,2)}\left( s_{2}\right) \asymp \frac{1}{%
n^{t_{2}}+n},
\end{equation*}
we have%
\begin{equation*}
1-H_{n}^{(1,2)}(s_{1},s_{2})+1-H_{n}^{(2,2)}(s_{2})\asymp
1-H_{n}^{(1,2)}(s_{1},s_{2})=Q_{n}^{(1,2)}(s_{1},s_{2}).
\end{equation*}%
This, in turn, yields for $T\leq \sqrt{n}$%
\begin{equation*}
C_{1}Q_{n}^{(1,2)}(s_{1},s_{2})L_{\mathrm{T}1}\leq
R_{2}(n,\mathbf{s})\leq
C_{2}Q_{n}^{(1,2)}(s_{1},s_{2})L_{\mathrm{T}}.
\end{equation*}%
From this estimate, (\ref{Lapla}) and Lemma \ref{L_twodimen} we get as $%
n\rightarrow \infty $
\begin{equation*}
\mathbb{E}\left[ 1-s_{1}^{Z_{n1}}s_{2}^{Z_{n2}}\right] \sim \frac{K_{0}}{%
C(t_{1},t_{2})}\log n,
\end{equation*}%
where
\begin{equation*}
C(t_{1},t_{2})=\left\{
\begin{array}{ccc}
1/2 & \text{if} & t_{1}\in \left( 0,\infty \right) ,0<t_{2}\leq 1; \\
t_{2}/2 & \text{if} & t_{1}\in \left( 0,\infty \right) ,1<t_{2}< 2; \\
1 & \text{if} & 0<t_{1}<1,t_{2}\geq 2; \\
1+\min \left( t_{1}-1,t_{2}-2\right) & \text{if} & t_{1}\geq 1,t_{2}\geq 2.%
\end{array}%
\right.
\end{equation*}%
Since $\mathbb{P}\left( \mathbf{Z}_{n}\neq \mathbf{0}\right) \sim
2K_{0}\left( \log n\right) ^{-1}$ for $N=2$, we conclude that for
positive $t_{1}$ and $t_{2}$%
\begin{eqnarray*}
\lim_{n\rightarrow \infty }\mathbb{E}\left[ s_{1}^{Z_{n1}}s_{2}^{Z_{n2}}\Big|%
\mathbf{Z}_{n}\neq \mathbf{0},\text{ }X_{0}=1,\mathbf{Z}_{0}=\mathbf{0}%
\right] &=&1-\lim_{n\rightarrow \infty }\frac{\mathbb{E}\left[
1-s_{1}^{Z_{n1}}s_{2}^{Z_{n2}}\right] }{\mathbb{P}\left(
\mathbf{Z}_{n}\neq
\mathbf{0}\right) } \\
&=&1-\frac{1}{2C(t_{1},t_{2})}=A\left( t_{1},t_{2}\right) .
\end{eqnarray*}%
Hence the statement of Theorem \ref{T_twodim}\ follows in an
ordinary way. \end{proof}

\textbf{Acknowledgement}

 VV was supported in part by the Russian Foundation
for Basic Research Project N~14-01-00318 and by CNRS of France for
a scientific stay of three months in 2013, at LMBA, UMR 6205,
Univ. Bretagne-Sud, where the present work has essentially been
done.


%
%
%
%

\end{document}